\documentclass[a4paper]{amsart}

\usepackage{amsfonts}
\usepackage{amssymb}
\usepackage{enumerate}
\usepackage{amsmath}
\usepackage{amsthm}
\usepackage{graphicx}

\numberwithin{equation}{section}

\newtheorem{theorem}{Theorem}[section]
\newtheorem{lemma}[theorem]{Lemma}

\newtheorem{corollary}[theorem]{Corollary}

\theoremstyle{definition}

\newtheorem{remark}[theorem]{Remark}
\newtheorem{definition}[theorem]{Definition}

\newtheorem*{MT}{Theorem \ref{t:mt}}

\DeclareMathOperator{\diam}{diam}
\DeclareMathOperator{\dist}{dist}

\DeclareMathOperator{\cl}{cl}

\title{Generalized Hausdorff measure for generic compact sets}

\author{Rich\'ard Balka}
\address{Alfr\'ed R\'enyi Institute of Mathematics, Hungarian Academy of Sciences,
PO Box 127, 1364 Budapest, Hungary}
\email{balka.richard@renyi.mta.hu}
\thanks{We gratefully acknowledge the support of the
Hungarian Scientific Foundation grant no.~72655. The second author was also supported by the EPSRC grant EP/G050678/1.}

\author{Andr\'as M\'ath\'e}
\address{Mathematics Institute, University of Warwick,
Coventry, CV4~7AL, United Kingdom}
\email{A.Mathe@warwick.ac.uk}

\date{}

\begin{document}

\begin{abstract} Let $X$ be a Polish space. We prove that the generic compact set
$K\subseteq X$ (in the sense of Baire category) is either finite or there is a continuous gauge
function $h$ such that
$0<\mathcal{H}^{h}(K)<\infty$, where $\mathcal{H}^h$ denotes the $h$-Hausdorff measure.
This answers a question of C.~Cabrelli, U.~B.~Darji, and U.~M.~Molter. Moreover, for every weak
contraction $f\colon K\to X$ we have $\mathcal{H}^{h} \left(K\cap f(K)\right)=0$.
This is a measure theoretic analogue of a result of M.~Elekes.
\end{abstract}



\maketitle

\section{Introduction}

Hausdorff dimension is one of the most important concepts to measure the size of a metric space, but
there are some cases when a finer notion of dimension is needed. An important
example is the trail of the $n$-dimensional $(n\geq 2)$ Brownian motion defined on $[0,1]$.
It has Hausdorff dimension $2$ almost surely, but its $\mathcal{H}^2$ measure is $0$ with probability $1$. It is well-known
that there is a gauge function $h$ such that the $h$-Hausdorff measure of the trail is positive and finite almost surely, where
$h(x)=x^2 \log \log (1/x)$ if  $n\geq 3$ and $h(x)=x^2 \log (1/x) \log \log \log (1/x)$ if $n=2$.
Thus the exact dimension is logarithmically smaller than $2$.

R. O. Davies \cite{D} constructed a Cantor set $K\subseteq \mathbb{R}$ that is either null or non-$\sigma$-finite for every translation
invariant Borel measure on $\mathbb{R}$. This implies that there is no gauge function $h$ such that $0<\mathcal{H}^{h}(K)<\infty$,
where $\mathcal{H}^h$ denotes the $h$-Hausdorff measure. C.~Cabrelli, U.~B.~Darji, and U.~M.~Molter \cite{C}
dealt with the problem that for `how many' compact sets $K\subseteq \mathbb{R}$ exist a translation invariant
Borel measure $\mu$ or a gauge function $h$ such that $0<\mu(K)<\infty$ or $0<\mathcal{H}^{h}(K)<\infty$, respectively.
They proved that the generic compact set $K\subseteq \mathbb{R}$
(see Definition \ref{d:generic})
admits a translation invariant Borel measure $\mu$ such that $0<\mu(K)<\infty$.
They defined a compact set $K\subseteq \mathbb{R}$ to be \emph{$\mathcal{H}$-visible}
if there is a gauge function $h$ such that $0<\mathcal{H}^h(K)<\infty$.
They showed that the set of $\mathcal{H}$-visible compact sets is dense in the space of all non-empty compact subsets of $\mathbb{R}$ endowed with the Hausdorff metric.
They posed the problem whether the generic compact set $K\subseteq \mathbb{R}$ is $\mathcal{H}$-visible.
We answer this question affirmatively by the following more general result.
\begin{theorem} \label{t:gg} Let $X$ be a Polish space. The generic compact set $K\subseteq X$ is either finite or there is a continuous gauge function
$h$ such that $0<\mathcal{H}^{h}(K)<\infty$.
\end{theorem}
We remark here that for every fixed gauge function $h$
the generic compact set $K\subseteq X$ has zero $\mathcal{H}^h$ measure.

If $X$ is a perfect Polish space then the set of finite compact subsets of $X$ form a meager set in the metric space of all non-empty compact subsets of
$X$ endowed with the Hausdorff metric. Therefore Theorem~\ref{t:gg} implies the following result.

\begin{corollary} Let $X$ be a perfect Polish space. For the generic compact set $K\subseteq X$ there is a
continuous gauge function $h$ such that $0<\mathcal{H}^{h}(K)<\infty$.
\end{corollary}

M.~Elekes \cite{E} studied metric spaces $X$ which are not complete but possess the Banach Fixed Point Theorem, that is,
every contraction $f\colon X\to X$ has a fixed point. He proved the following theorem which is interesting in its own right.

\begin{theorem}[M. Elekes] \label{t:E} For the generic compact set $K\subseteq \mathbb{R}$ for any contraction $f\colon K\to \mathbb{R}$
the set $f(K)$ does not contain a non-empty relatively open subset of $K$.
\end{theorem}

The first author of the present paper \cite{B} constructed metric spaces $X$ such that every weak contraction
$f\colon X\to X$ is constant, where he used
measure theoretic methods. Based on \cite{B}, we prove the (somewhat stronger) measure theoretic analogue
of Theorem \ref{t:E}.

\begin{MT}[Main Theorem]\label{t:hh}
 Let $X$ be a Polish space. The generic compact set $K\subseteq X$ is either finite or
there is a continuous gauge function $h$ such that $0<\mathcal{H}^{h}(K)<\infty$, and for every weak contraction $f\colon
K\to X$ we have $\mathcal{H}^{h} \left(K\cap f(K)\right)=0$.
\end{MT}

In Section~\ref{s:prelim} we recall some notions from metric spaces which we use in this paper.
In Section~\ref{s:K} we introduce the notion of balanced compact sets. It is shown in \cite{B} that for every balanced compact set there is a continuous gauge function $h$ such that
$0<\mathcal{H}^{h}(K)<\infty$ and that $\mathcal{H}^{h} \left(K\cap f(K)\right)=0$ for every weak contraction $f\colon K\to X$. In Section~\ref{s:main} we prove that in a perfect Polish space the generic compact set is a balanced compact set, and we conclude the proof of Theorem~\ref{t:mt} and Theorem~\ref{t:gg}.

\section{Preliminaries}\label{s:prelim}

Let $(X,d)$ be a metric space, and let $A,B\subseteq X$ be arbitrary
sets. We denote by $\cl A$ and $\diam A$ the closure and the diameter of $A$, respectively.
We use the convention $\diam \emptyset = 0$. The \emph{distance} of the sets $A$ and
$B$ is defined by $\dist (A,B)=\inf \{d(x,y): x\in A, \, y\in B\}$. Let $B(x,r)=\{y\in X: d(x,y)\leq r\}$
and $U(x,r)=\{y\in X: d(x,y)< r\}$ for all $x\in X$ and $r>0$.
More generally, consider $B(A,r)=\{x\in X: \dist (A,\{x\})\leq r\}$.

The function $h\colon [0,\infty)\to [0,\infty)$ is defined to be a \emph{gauge function} if it
is non-decreasing, right-continuous, and $h(x)=0$ iff $x=0$. For $A\subseteq X$ and $\delta>0$ consider
\begin{align*}
\mathcal{H}^{h}_{\delta}(A)&=\inf \left\{ \sum_{i=1}^\infty
h\left(\diam A_{i}\right): A \subseteq \bigcup_{i=1}^{\infty} A_{i},
\, \forall i \,
\diam A_i \leq \delta \right\}, \\
\mathcal{H}^{h}(A)&=\lim_{\delta\to 0+}\mathcal{H}^{h}_{\delta}(A).
\end{align*}
We call $\mathcal{H}^{h}$ the \emph{$h$-Hausdorff measure}. For more information on
these concepts see \cite{Ro}.

Let $X$ be a complete metric space. A set is \emph{somewhere
dense} if it is dense in a non-empty open set, otherwise it is
called \emph{nowhere dense}. We say that $M \subseteq X$ is
\emph{meager} if it is a countable union of nowhere dense sets, and
a set is \emph{co-meager} if its complement is meager.
Baire's Category Theorem implies that a set is co-meager if and only if
it contains a dense $G_{\delta}$ set. We
say that the \emph{generic} element $x \in X$ has property $\mathcal{P}$ if
$\{x \in X : x \textrm{ has property } \mathcal{P} \}$ is co-meager.
A metric space $X$ is \emph{perfect} if it has no isolated points.
A metric space $X$ is \emph{Polish} if it is complete and separable.

Given two metric spaces $(X,d_{X})$ and $(Y,d_{Y})$, a function
$f\colon X\to Y$ is called a \emph{weak contraction} if $d_{Y}(f(x_{1}),f(x_{2}))<
d_{X}(x_{1},x_{2})$ for every $x_{1},x_{2}\in X$, $x_1\neq x_2$.

Let $\mathbb{N}^{<\omega}$ stand for the set of finite sequences of natural numbers.
Let us denote the set of positive odd numbers by $2\mathbb{N}+1$.

\section{The definition of balanced compact sets} \label{s:K}

Following \cite{B} we define balanced compact sets.

\begin{definition} \label{d:J} If $a_n$ $(n\in \mathbb{N}^+)$ are positive integers
then let us consider, for every $n\in \mathbb{N}^+$,
$$\mathcal{I}_{n}=\prod_{k=1}^{n}\{1,2,\dots,a_k\} \quad \textrm{and} \quad \mathcal{I}=\bigcup_{n=1}^{\infty} \mathcal{I}_{n}.$$
We say that a map $\Phi \colon 2\mathbb{N} +1 \to \mathcal{I}$
is an \emph{index function  according to the sequence $\langle a_n \rangle$}
if it is surjective and $\Phi(n) \in \bigcup _{k=1}^{n} \mathcal{I}_{k}$ for every odd $n$.
\end{definition}

\begin{definition} \label{d:balanced} Let $X$ be a Polish space.
A compact set $K\subseteq X$ is \emph{balanced} if it is of the form
\begin{equation} \label{eq:defK} K=\bigcap _{n=1}^{\infty}\left(\bigcup_{i_1=1}^{a_1}\dots \bigcup_{i_n=1}^{a_n}C_{i_1 \dots i_n} \right),
\end{equation}
where the $a_{n}$ are positive integers and $C_{i_1\dots i_n}\subseteq X$ are non-empty closed sets with
the following properties. There are positive reals $b_n$ and there is an
index function $\Phi \colon 2 \mathbb{N}+1 \to \mathcal{I}$ according to the sequence $\langle a_n\rangle$
such that for all $n\in \mathbb{N}^+$ and $(i_1,\dots ,i_{n}),(j_1,\dots ,j_{n})\in \mathcal{I}_{n}$

\begin{enumerate}[(i)]
\item \label{01} $a_1\geq 2$ and $a_{n+1}\geq n a_{1}\cdots a_{n}$,
\item \label{02} $C_{i_{1}\dots i_{n+1}}\subseteq
C_{i_1 \dots i_{n}}$,
\item \label{03} $\diam C_{i_{1} \dots i_n}\leq b_n$,
\item  \label{04} $\dist(C_{i_1 \dots i_n},C_{j_1\dots j_n})>2b_n$
if $(i_1,\dots ,i_n)\neq (j_1, \dots ,j_n)$.
\item \label{05}
If $n$ is odd, $C_{i_1 \dots i_{n}}\subseteq C_{\Phi(n)}$ and $C_{j_1 \dots j_{n}}
\nsubseteq C_{\Phi(n)}$, then for all $s,t
\in\{1,\dots ,a_{n+1}\}$, $s\neq t$, we have
$$\dist\left(C_{i_1\dots i_{n}s},C_{i_1 \dots i_{n}t}\right)> \diam \left(\bigcup
_{j_{n+1}=1}^{a_{n+1}} C_{j_1 \dots j_{n}j_{n+1}}\right).$$
\end{enumerate}
\end{definition}

\begin{remark} Property~\eqref{05} and the notion of an index function $\Phi$ are not needed for the proof of Theorem~\ref{t:gg}, only for Theorem~\ref{t:mt}.

Note that we cannot require property~\eqref{05} for every positive integer.  The proof of Lemma~\ref{l:balanced} only works if we restrict this property to odd numbers.
\end{remark}

\begin{remark} In a countable Polish space $X$ there is no balanced compact set $K\subseteq X$, since
every balanced compact set has cardinality $2^{\aleph_0}$.
\end{remark}

\section{The Main Theorem}\label{s:main}

\begin{definition} \label{d:generic}
If $X$ is a Polish space then let $(\mathcal{K}(X),d_{H})$ be the set of non-empty compact subsets of
$X$ endowed with the \emph{Hausdorff metric};
that is, for each $K_1,K_2\in \mathcal{K}(X)$,
$$d_{H}(K_1,K_2)=\min \left\{r: K_1\subseteq B(K_2,r) \textrm{ and } K_2\subseteq B(K_1,r)\right\}.$$
It is well-known that
$(\mathcal{K}(X),d_{H})$ is a Polish space, see
e.g.~\cite{Ke}, hence we can use Baire category arguments.
Let $B_{H}(K,r)\subseteq \mathcal{K}(X)$ denote the closed ball around $K$ with radius $r$.
\end{definition}

The main goal of this paper is to prove the following theorem.

\begin{theorem}[Main Theorem] \label{t:mt}
 Let $X$ be a Polish space. The generic compact set $K\subseteq X$ is either finite or
there is a continuous gauge function $h$ such that $0<\mathcal{H}^{h}(K)<\infty$, and for every weak contraction $f\colon
K\to X$ we have $\mathcal{H}^{h} \left(K\cap f(K)\right)=0$.
\end{theorem}

\begin{remark}
If $X$ is a Polish space and $h$ is a fixed gauge function then
it is easy to see that for the generic compact set $K\subseteq X$ we have $\mathcal{H}^{h}(K)=0$.
If $X$ is uncountable then infinite compact sets form a second category subset in $\mathcal{K}(X)$,
therefore the gauge function $h$ must depend on $K$ in the Main Theorem.
\end{remark}

The first author of the paper proved the following theorem \cite[Thm. 5.1]{B}.

\begin{theorem} \label{t:old} Let $X$ be a Polish space, and let $K\subseteq X$ be a balanced compact set.
Then there exists a continuous gauge function $h$ such that
$0<\mathcal{H}^{h}(K)<\infty$, and for every weak contraction $f\colon K\to X$ we have $\mathcal{H}^{h}
\left(K\cap f(K)\right)=0$.
\end{theorem}

If $h$ is a gauge function then finite sets have zero $\mathcal{H}^h$ measure, so Theorem \ref{t:old} also
holds for compact sets $K\subseteq X$ that can be written as a union of a balanced compact set and a finite set.
Therefore the following theorem implies our Main Theorem.

\begin{theorem} \label{t:g} If $X$ is a Polish space then the generic
compact set $K\subseteq X$ is either finite or it can be written as the union of a balanced compact set and a finite set.
\end{theorem}

To prove Theorem~\ref{t:g} first we give definitions and prove two  key lemmas.

\begin{definition} \label{d:Phi} Let us fix an onto map $\Psi \colon 2\mathbb{N}+1\to \mathbb{N}^{<\omega}$ such that $\Psi(n)$ has at most $n$ coordinates for every odd $n$.

For $n\in \mathbb{N}^{+}$ and sequence $(a_1, a_2,\ldots, a_{2n-1})$, we define the function
$$\Phi=\Phi_{a_1 a_2 \ldots a_{2n-1}} :
\{2k-1: 1\leq k\leq n\}\to \bigcup _{m=1}^{2n-1} \mathcal{I}_{m}$$ by setting
$$\Phi(2k-1)=
\begin{cases} \Psi(2k-1) & \textrm{ if } \Psi(2k-1)\in \bigcup_{m=1}^{2k-1} \mathcal{I}_{m} \\
1\in \mathcal{I}_{1} & \textrm{ otherwise}.
\end{cases}
$$
\end{definition}

\begin{remark} \label{r:index}
If $\langle a_n \rangle_{n\in \mathbb{N}^+}$ is a sequence of positive integers then the above definition implies that the functions
$\Phi_{a_1 \dots a_{2n-1}}$ have a common extension $\Phi \colon 2\mathbb{N} +1 \to \mathcal{I}$,
and $\Phi$ is an index function according to the sequence $\langle a_n \rangle$.
\end{remark}

 Let $X$ be a Polish space.

\begin{definition}\label{def:balancedscheme}
Let $n\in \mathbb{N}^+$.
We call the pair of $(a_1, \ldots, a_{2n})$ and
$$\left\{\big((i_1,\ldots, i_k), \,U_{i_1\ldots i_k}\big) \,:\, (i_1, \dots, i_k ) \in \mathcal{I}_k, \ 1\le k \le  2n \right\}$$
a \emph{balanced scheme} of size $n$ if the numbers $a_k$ are positive integers, the sets $U_{i_1\ldots i_k}$ are non-empty open subsets of $X$, and there exist positive reals $b_k$ for which

\begin{enumerate}[(1)]
\item \label{x1} $a_1\geq 2$ and $a_{k}\geq (k-1)a_{1}\cdots a_{k-1}$ for all $2\le k \leq 2n$,
\item \label{x2} $\cl U_{i_{1}\dots i_{k}} \subseteq
U_{i_1 \dots i_{k-1}}$ for all $(i_1,\dots, i_k) \in \mathcal{I}_k$ and $2\le k\leq 2n$,
\item \label{x3} $\diam U_{i_{1} \dots i_k}\leq b_k$ for all $(i_1, \dots ,i_k) \in \mathcal{I}_k$ and $1\le k\le 2n$,
\item  \label{x4} $\dist(U_{i_1 \dots i_k}, \,U_{j_1\dots j_k})>2b_k$
if $(i_1,\dots, i_k)\neq (j_1, \dots ,j_k) \in \mathcal{I}_k$ and $1\le k\le 2n$.
\item \label{x5} Let $\Phi=\Phi_{a_1 \ldots a_{2n-1}}$. If $k<2n$ is odd, $U_{i_1 \dots i_{k}}\subseteq U_{\Phi(k)}$
and $U_{j_1 \dots j_{k}} \nsubseteq U_{\Phi(k)}$, then for all $s,t
\in\{1,\dots ,a_{k+1}\}$, $s\neq t$, we have
$$\dist\left(U_{i_1\dots i_{k}s},U_{i_1 \dots i_{k}t}\right)> \diam \left(\bigcup
_{j_{k+1}=1}^{a_{k+1}} U_{j_1 \dots j_{k}j_{k+1}}\right).$$

Let $(\emptyset, \emptyset)$ be the \emph{balanced scheme} of size $0$.
\end{enumerate}

\end{definition}

\begin{definition}\label{d:upi}
If $n\in\mathbb{N}^+$ and $\pi$ is a balanced scheme of size $n$ as in Definition~\ref{def:balancedscheme}, then we define a non-empty open subset of $\mathcal{K}(X)$,
\begin{equation*} \mathcal{U}(\pi) \! = \!  \left\{K\in \mathcal{K}(X)  :  K\subseteq \! \bigcup_{i_1=1}^{a_1} \!  \cdots \! \bigcup_{i_{2n}=1}^{a_{2n}}
U_{i_1 \ldots i_{2n}}, \ \forall (i_1, \ldots ,i_{2n})\in \mathcal{I}_{2n} \  K\cap U_{i_1 \ldots i_{2n}}\neq \emptyset \right\}.
\end{equation*}
For $\pi=(\emptyset, \emptyset)$ we define $\mathcal{U}(\pi)=\mathcal{K}(X)$.

Assume $n\in \mathbb{N}$, and let $\pi$ and $\pi'$ be balanced schemes of size $n$ and $n+1$, respectively. We say that $\pi'$ is \emph{consistent} with
$\pi$ if $a_k(\pi')=a_k(\pi)$ and $U_{i_1\dots i_k}(\pi')=U_{i_1\dots i_k}(\pi)$ for all $k\in \{1,\dots, 2n\}$ and $(i_1,\dots,i_k)\in \mathcal{I}_{k}$.
\end{definition}

\begin{remark} Let $\pi$ and $\pi'$ be balanced schemes of size $n$ and $n+1$, respectively.
If $\pi'$ is consistent with $\pi$ then $\mathcal{U}(\pi')\subseteq \mathcal{U}(\pi)$, and we may assume $b_k(\pi')=b_k(\pi)$ for every $k\in \{1,\dots, 2n\}$.
\end{remark}

\begin{lemma} \label{l:balanced} Assume $n\in \mathbb{N}$. Let $X$ be a non-empty perfect Polish space, let $\pi$ be a balanced scheme of size $n$,
and let $\mathcal{V}\subseteq \mathcal{U}(\pi)$ be a non-empty open subset of $\mathcal{K}(X)$. There exists a balanced scheme $\pi'$ of size $n+1$ such that $\pi'$ is consistent with $\pi$ and
$\mathcal{U}(\pi')\subseteq \mathcal{V}$. \end{lemma}

\begin{proof}[Proof of Lemma \ref{l:balanced}]
Let $a_k(\pi')=a_k(\pi)=a_k$, $b_k(\pi')=b_k(\pi)=b_k$, $U_{i_1\dots i_k}(\pi')= U_{i_1\dots i_k}(\pi)=U_{i_1\dots i_k}$ for every $k\leq 2n$ and
$(i_1,\dots,i_k)\in \mathcal{I}_{k}$. Then $\pi'$ will satisfy properties \eqref{x1}-\eqref{x5} for
all $k\leq 2n$, since the map $\Phi_{a_1 \dots a_{2n+1}}$ extends $\Phi_{a_1 \dots a_{2n-1}}$ by Definition \ref{d:Phi}.
Therefore it is enough to construct $a_{k}(\pi')=a_{k}$, $b_{k}(\pi')=b_{k}$,
and $U_{i_1\dots i_{k}}(\pi')=U_{i_1\dots i_{k}}$ for $k\in \{2n+1,2n+2\}$ and $(i_1,\dots,i_{k})\in \mathcal{I}_{k}$.

As finite compact sets form a dense subset in $\mathcal{K}(X)$ and $X$ is perfect,
it is easy to see that there is a finite set $K_0\in \mathcal{V}$ with the following property. There is an integer
$N\geq 2$ such that $N\geq 2n(a_1\cdots a_{2n})$ and $\# (K_0\cap U_{i_1\dots i_{2n}})=N$ for every
$(i_1,\dots,i_{2n})\in \mathcal{I}_{2n}$. Set $a_{2n+1}=N$, then \eqref{x1} holds for $k=2n+1$. For $(i_1,\dots,i_{2n})\in \mathcal{I}_{2n}$ let
\begin{equation*} \label{e1}
K_0\cap U_{i_1\dots i_{2n}}=\left\{x_{i_1\dots i_{2n+1}}: 1\leq i_{2n+1}\leq a_{2n+1}\right\}.
\end{equation*}
For $(i_1,\dots ,i_{2n+1})\in \mathcal{I}_{2n+1}$ consider the non-empty open sets
\begin{equation*} \label{e2}
U_{i_1\dots i_{2n+1}}=U(x_{i_1\dots i_{2n+1}},b_{2n+1}/2),
\end{equation*}
where $b_{2n+1}>0$ is sufficiently small. Then the sets $U_{i_1\dots i_{2n+1}}$ satisfy properties \eqref{x2}--\eqref{x4}, and
$B_{H}(K_0,b_{2n+1})\subseteq \mathcal{V}$.
(Notice that we did not require property~\eqref{x5} to hold for even numbers, and indeed, we could not satisfy it here for an arbitrary $\mathcal{V}$.)

Let $a_{2n+2}=(2n+1)(a_1\cdots a_{2n+1})$, so \eqref{x1} holds for $k=2n+2$.
First consider those $(i_1, \ldots, i_{2n+1})$ for which $U_{i_1\dots i_{2n+1}}\subseteq U_{\Phi(2n+1)}$, where
$\Phi=\Phi_{a_1 \ldots a_{2n+1}}$. Then by the perfectness of $X$
we can fix distinct points $x_{i_1\dots i_{2n+2}}\in U_{i_1\dots i_{2n+1}}$  ($i_{2n+2}\in \{1,\dots,a_{2n+2}\}$).

Let $\delta$ be the minimum distance between the points $x_{i_1 \ldots i_{2n+2}}$ we have defined so far. Now consider those $(i_1, \ldots, i_{2n+1})$ for which $U_{i_1\dots i_{2n+1}}\nsubseteq U_{\Phi(2n+1)}$. For each of them, fix distinct points $x_{i_1\dots i_{2n+2}}\in U_{i_1\dots i_{2n+1}}$ ($i_{2n+2}\in \{1,\dots,a_{2n+2}\}$) such that
$$ \diam \left(\bigcup_{i_{2n+2}=1}^{a_{2n+2}} \{x_{i_1\dots i_{2n+2}}\}\right)\leq \frac{\delta}{2}.$$
For $(i_1,\dots ,i_{2n+2})\in \mathcal{I}_{2n+2}$ consider the non-empty open sets
\begin{equation*} \label{e3}
U_{i_1\dots i_{2n+2}}=U(x_{i_1\dots i_{2n+2}},b_{2n+2}/2),
\end{equation*}
where $b_{2n+2}>0$ is sufficiently small. Then the sets $U_{i_1\dots i_{2n+2}}$ satisfy properties \eqref{x2}--\eqref{x5}. Therefore $\pi'$ is a balanced scheme of size $n+1$,
and $\pi'$ is consistent with $\pi$.

Finally, we need to prove that $\mathcal{U}(\pi')\subseteq \mathcal{V}$. We show that for every $K\in \mathcal{U}(\pi')$,
\begin{equation} \label{dH} d_{H}(K,K_0)\leq b_{2n+1}.
\end{equation}
Let $K\in \mathcal{U}(\pi')$. By the definition of $\mathcal{U}(\pi')$ we have $K \subseteq \bigcup_{i_1=1}^{a_1}   \cdots  \bigcup_{i_{2n+1}=1}^{a_{2n+1}}
U_{i_1 \ldots i_{2n+1}}$ and $K\cap U_{i_1\dots i_{2n+1}}\neq \emptyset$ for all $(i_1,\dots ,i_{2n+1})\in \mathcal{I}_{2n+1}$.
The set $K_0$ has the above properties by its definition, too. As $\diam  U_{i_1\dots i_{2n+1}}\leq b_{2n+1}$ for all $(i_1,\dots ,i_{2n+1})\in \mathcal{I}_{2n+1}$,
\eqref{dH} follows. Equation \eqref{dH} implies $\mathcal{U}(\pi')\subseteq  B_{H}(K_0,b_{2n+1})$,
therefore $B_{H}(K_0,b_{2n+1})\subseteq \mathcal{V}$ yields $\mathcal{U}(\pi')\subseteq \mathcal{V}$.
\end{proof}

\begin{lemma} \label{cor} Assume $n\in \mathbb{N}$.
Let $X$ be a non-empty perfect Polish space, and let $\pi$ be a balanced scheme of size $n$. Then there are balanced schemes $\pi_j$ $(j\in \mathbb{N})$ of size $n+1$
such that each $\pi_j$ is consistent with $\pi$, the sets $\mathcal{U}(\pi_j)$ $(j\in \mathbb{N})$ are pairwise disjoint,
and $\bigcup_{j=0}^{\infty} \mathcal{U}(\pi_j)$ is dense in $\mathcal{U}(\pi)$.
\end{lemma}

\begin{proof}
Let $\mathcal{U}_{i}\subseteq \mathcal{U}(\pi)$ $(i\in \mathbb{N})$ be non-empty disjoint open sets such that $\bigcup_{i=0}^{\infty} \mathcal{U}_i$ is dense in $\mathcal{U}(\pi)$. For all $i\in \mathbb{N}$ let
$\mathcal{B}_i$ be a countable basis of $\mathcal{U}_i$, and let $\mathcal{B}=\bigcup_{i=0}^{\infty} \mathcal{B}_{i}$.
We may assume $\emptyset\notin \mathcal{B}$ and let us consider an enumeration $\mathcal{B}=\{\mathcal{V}_n: n\in \mathbb{N}\}$.
Let $j\in \mathbb{N}$ and assume that $\pi_k$ and $n(k)\in \mathbb{N}$ ($k<j$) are already defined such that $\mathcal{U}(\pi_k)\subseteq \mathcal{V}_{n(k)}$ for $k<j$. Consider
$$n(j)=\min \left\{n\in \mathbb{N}: \mathcal{V}_n\cap \left(\cup_{k<j}  \mathcal{U}(\pi_{k})\right)=\emptyset\right\}.$$
The definition of $\mathcal{B}$ and the induction hypothesis easily imply that $\bigcup_{k<j}  \mathcal{U}(\pi_{k})$ can intersect at most $j$ open sets $\mathcal{U}_i$,
so $n(j)<\infty$ exists. Lemma~\ref{l:balanced} implies that there is a balanced scheme $\pi_j$ of size $n+1$
such that $\pi_j$ is consistent with $\pi$ and $\mathcal{U}(\pi_j)\subseteq \mathcal{V}_{n(j)}$.

The construction yields that $\bigcup_{j=0}^{\infty} \mathcal{U}(\pi_j)$ intersects each $\mathcal V_i$, thus it is dense in each $\mathcal{U}_i$,
therefore it is dense in $\mathcal{U}(\pi)$, and the union is clearly a disjoint union.
\end{proof}

Now we are ready to prove Theorem \ref{t:g} that implies our Main Theorem.

\begin{proof}[Proof of Theorem \ref{t:g}] First assume that $X$ is perfect,
we prove that the generic compact set $K\subseteq X$ is balanced.
We may assume that $X\neq \emptyset$.
Let $\mathcal{G}_0=\mathcal{K}(X)$. Lemma~\ref{cor} implies that
there are balanced schemes $\pi_j$ $(j\in \mathbb{N}$) of size $1$ such that the disjoint union
$$\mathcal{G}_{1}=\bigcup_{j_1=0}^{\infty} \mathcal{U}(\pi_{j_1})$$
is a dense open set in $\mathcal{K}(X)$. Assume by induction that the balanced schemes $\pi_{j_1\dots j_n}$ of size $n$ and the dense open set
$\mathcal{G}_{n}$ are already defined. Lemma~\ref{cor} implies that for every $j_1,\dots,j_{n}\in \mathbb{N}$ there exist balanced schemes
$\pi_{j_1\dots j_{n+1}}$ $(j_{n+1}\in \mathbb{N})$ of size $n+1$ such that $\pi_{j_1\dots j_{n+1}}$ is consistent with $\pi_{j_1\dots j_{n}}$ and the disjoint union
$\bigcup_{j_{n+1}=0}^{\infty} \mathcal{U}(\pi_{j_1\dots j_{n+1}})$ is dense in
$\mathcal{U}(\pi_{j_1\dots j_{n}})$.
Then the disjoint union
$$\mathcal{G}_{n+1}=\bigcup_{j_1=0}^{\infty} \cdots \! \bigcup_{j_{n+1}=0}^{\infty} \mathcal{U}(\pi_{j_1 \dots j_{n+1}})$$
is dense in $\mathcal{G}_{n}$, and the induction hypothesis yields that $\mathcal{G}_{n+1}$ is a dense open set in $\mathcal{K}(X)$.
Consider
$$\mathcal{G}=\bigcap_{n=0}^{\infty} \mathcal{G}_{n}.$$
As a countable intersection of dense open sets $\mathcal{G}$ is co-meager in $\mathcal{K}(X)$. Let $K\in \mathcal{G}$ be arbitrary fixed,
it is enough to prove that $K$ is balanced.
Since the $n$th level open sets $\mathcal{U}(\pi_{j_1\dots j_n})$ are pairwise disjoint, there is a (unique) sequence
$\langle j_{n} \rangle_{n\in \mathbb{N}^+}$ such that $K\in \mathcal{U}(\pi_{j_1\dots j_n})$ for all $n\in \mathbb{N}^+$. As the balanced scheme
$\pi_{j_1\dots j_{n+1}}$ is consistent with $\pi_{j_1\dots j_{n}}$ for every $n\in \mathbb{N}^+$,
there are positive integers $a_n$ and non-empty open sets
$U_{i_1\dots i_n}$ witnessing this fact.
By Remark~\ref{r:index},
the functions
$\Phi_{a_1 a_2 \ldots a_{2n-1}}$
have a common extension $\Phi \colon 2\mathbb{N}+1 \to \mathcal{I}$,
and $\Phi$ is an index function according to the sequence $\langle a_n \rangle$.
For
$n\in \mathbb{N}^+$ and $(i_1,\dots,i_n)\in \mathcal{I}_{n}$
let us define
$$C_{i_1\dots i_n}=\cl U_{i_1\dots i_n}.$$
Since $K\in \mathcal{U}(\pi_{j_1\dots j_n})$ for every $n$, Definition~\ref{d:upi} implies that
$$K=\bigcap _{n=1}^{\infty}\left(\bigcup_{i_1=1}^{a_1}\cdots \bigcup_{i_n=1}^{a_n}C_{i_1 \dots i_n} \right).$$
From Definition~\ref{def:balancedscheme} it follows that the positive integers $a_n$ and the non-empty closed sets $C_{i_1 \dots i_n}$ satisfy properties $\eqref{01}$--$\eqref{05}$
of Definition~\ref{d:balanced}.
Therefore $K$ is balanced.

Now let $X$ be an arbitrary non-empty Polish space.
Then there is a perfect set $X^{*}\subseteq X$ such that $U=X\setminus X^{*}$ is countable open, see \cite[(6.4) Thm.]{Ke}.
Let $S$ be the set of isolated points of $X$. Then $S$ is open, and $S\subseteq U$. We claim that $S$ is dense in $U$, thus $U\subseteq \cl S$.
Indeed, assume to the contrary that there is a non-empty open set $V\subseteq U$ such that $V\cap S=\emptyset$.
By shrinking $V$, we may suppose that $\cl V \subseteq U$. Then $\cl V\subseteq U$ is a non-empty perfect set, so it has
cardinality $2^{\aleph_0}$ by \cite[(6.3)~Cor.]{Ke}, which is a contradiction.

For a set $A\subseteq X$ let us denote by $\mathcal{K}(A)$ the metric space of non-empty compact subsets of $A$, similarly as in Definition \ref{d:generic}.

Since $S$ is open, compact non-empty subsets of $S$ form a dense open subset of $\mathcal{K}(\cl S)$. As $S$ is the set of isolated points, every compact subset of $S$ is finite.

The first part of the proof implies that there is a dense $G_{\delta}$ set $\mathcal{F}^{*}\subseteq \mathcal{K}(X^{*})$ such that every $K^{*}\in \mathcal{F}^{*}$ is balanced.

Let $\mathcal{F}\subseteq \mathcal{K}(X)$ be the the set of those non-empty compact subsets $K\subseteq X$ for which $K\cap \cl S\subseteq S$ and $K\cap X^* \in \mathcal{F}^* \cup \{\emptyset\}$.
Clearly, every $K\in\mathcal{F}$ is a union of $\emptyset$ or a balanced compact set in $X^*$ and finitely many points in $S$.
We claim that $\mathcal{F}$ is a dense $G_\delta$ subset of $\mathcal{K}(X)$. Let us define the continuous map
$$R\colon \mathcal{K}(X) \to \mathcal{K}(X^{*}) \cup\{\emptyset\}, \quad R(K)=K\cap X^{*},$$
where the distance of $\emptyset$ to points of $\mathcal{K}(X^{*})$ is defined to be $1$.

We show that the map $R$ is open. Let $K\in \mathcal{K}(X)$ and $C^{*}\in \mathcal{K}(X^{*}) \cup\{\emptyset\}$ be arbitrary, and set $K^{*}=K\cap X^{*}$.
It is enough to construct $C\in \mathcal{K}(X)$ such that $C\cap X^{*}=C^{*}$ and $d_{H}(K,C)\leq d_{H}(K^{*},C^{*})$.
If $K\subseteq X^{*}$ or $K^{*}=C^{*}$, then $C=C^{*}$ or $C=K$ works, respectively. Thus we may assume that $K\setminus X^{*}\neq \emptyset$ and $d_{H}(K^{*},C^{*})>0$.
The compactness of $K$ implies that there are finitely many
open sets $V_i$ such that $K\setminus X^{*}\subseteq \bigcup_{i=1}^{m} V_i$, $V_i\cap (K\setminus X^{*})\neq \emptyset$, and $\diam V_i \leq d_{H}(K^{*},C^{*})$ for all
$i\in \{1,\dots,m\}$. Let us choose $x_i\in V_i\setminus X^{*}$ for all $i\in \{1,\dots, m\}$ arbitrarily, and consider $C=C^{*}\cup \bigcup_{i=1}^{m} \{x_i\}$. It is easy to see that
$C\in \mathcal{K}(X)$ fulfills the required properties.

Since $R$ is open, $R^{-1}(\mathcal{F}^*\cup\{\emptyset\})$ is dense $G_\delta$ in $\mathcal{K}(X)$. We clearly have
$$\mathcal{F}=R^{-1}(\mathcal{F}^*\cup\{\emptyset\}) \cap \mathcal{K}((X\setminus \cl S)\cup S).$$
As $(X\setminus \cl S)\cup S$ is dense open in $X$, $\mathcal{K}((X\setminus \cl S)\cup S)$ is dense open in $\mathcal{K}(X)$. Thus $\mathcal{F}$ is dense $G_\delta$ in $\mathcal{K}(X)$,
which concludes the proof.
\end{proof}

\noindent \textbf{Acknowledgement.} The authors are indebted to
M.~Elekes and to an anonymous
referee for their valuable comments.


\noindent

\begin{thebibliography}{99}

\bibitem{B} R.~Balka, Metric spaces admitting only trivial weak contractions, arXiv:1202.1539,
to appear in \textit{Fund. Math.}

\bibitem{C} C.~Cabrelli, U.~B.~Darji, U.~M.~Molter, Visible and invisible Cantor sets, 
\textit{Excursions in Harmonic Analysis, Volume 2: The February Fourier Talks at the Norbert Wiener Center}, edited by 
Travis D. Andrews, Radu Balan, John J. Benedetto, Wojciech Czaja and Kasso A. Okoudjou, 11--22, Springer, 2013.

\bibitem{D} R.~O.~Davies, Sets which are null or non-sigma-finite for every translation-invariant measure,
\textit{Mathematika} \textbf{18} (1971), 161--162.

\bibitem{E} M.~Elekes, On a converse to Banach's Fixed Point
Theorem, \textit{Proc. Amer. Math. Soc.} \textbf{137} (2009), no. 9,
3139--3146.

\bibitem{Ke} A.~S.~Kechris, \textit{Classical descriptive set theory}, Springer-Verlag, 1995.

\bibitem{Ro} C.~A.~Rogers, \textit{Hausdorff measures},
Cambridge University Press, 1970.

\end{thebibliography}
\end{document}